\documentclass[11pt]{amsart}
\usepackage{latexsym}
\usepackage{epsfig}
\usepackage{graphicx}
\usepackage{amsmath}
\usepackage{amssymb}
\usepackage{mathrsfs}
\usepackage{eucal}
\usepackage{bm}
\usepackage{color}
\newtheorem{proposition}{Proposition}[section]
\newtheorem{theorem}[proposition]{Theorem}
\newtheorem{corollary}[proposition]{Corollary}
\newtheorem{lemma}[proposition]{Lemma}

\theoremstyle{definition}
\newtheorem{definition}[proposition]{Definition}

\newtheorem{remark}[proposition]{Remark}

\numberwithin{equation}{section}

\def\R{\mathbb R}
\def\Dx{\Delta_x}
\def\Nx{\nabla_x}
\def\Dt{\partial_t}
\def\({\left(}
\def\){\right)}
\def\eb{\varepsilon}

\def\Bbb{\mathbb}

\def\qand{\quad\mbox{and}\quad}

\title[]{Global well-posedness in  uniformly local spaces for the Cahn-Hilliard equation in $\R^3$}
 \author[]
{Jon Pennant and Sergey Zelik}

\thanks{
The authors would like to thank Alain Miranville for fruitful
comments.}

\address{Department of mathematics,
\newline\indent
University of Surrey Guildford,
\newline\indent  GU2 7XH United Kingdom }
\email{J.Pennant@surrey.ac.uk}
\email{S.Zelik@surrey.ac.uk}


\begin{document}

\begin{abstract}
We study  the infinite-energy solutions of the Cahn-Hilliard equation in the whole 3D space
in uniformly local phase spaces. In particular, we establish the global existence of solutions  for the case of regular potentials of
arbitrary polynomial growth and for the case of sufficiently strong singular potentials. For these cases, the uniqueness and further regularity of the obtained solutions are proved as well. We discuss also  the analogous problems for the case of the so-called Cahn-Hilliard-Oono equation
where, in addition, the dissipativity of the associated solution semigroup is established.
\end{abstract}

\maketitle
\tableofcontents
\section{Introduction}\label{s.int}
We study the classical Cahn-Hilliard equation
\begin{equation}\label{00.ch}
\Dt u=\Dx(-\Dx u+f(u)+g)
\end{equation}
considered in the whole space $x\in\R^3$.
\par
It is well-known that  the Cahn-Hillard (CH) equation is central for
the material sciences and extensive amount of papers are devoted to
the mathematical analysis of this equation and various of its
generalizations.  In
particular, in the case where the underlying domain $\Omega$ is bounded,
 the analytic and dynamic properties of the CH equations are relatively well-understood including the questions of well-posedness
 (even in the case of singular potentials $f$) and dissipativity, smoothness, existence of global and exponential
 attractors, upper and lower bounds for the dimension, etc.
 We mention here only some contributors, namely, \cite{CH,D,EGZ,EMZ1,EGW,GSZ,GSZ1,GPS,Ka,EK1,MZ3,MZ2,MZ1,No1,No,WW}
 (see also the references therein).
\par
The situation in the case where the underlying domain is unbounded
is essentially less clear even in the case of finite-energy
solutions. Indeed, the key feature of the CH equation in bounded
domains which allows to build up a reasonable theory (especially in
 the case of rapidly growing or singular nonlinearities) is the possibility to obtain good estimates in the negative Sobolev space $W^{-1,2}(\Omega)$
  and, to this end, one should use the inverse Laplacian $(-\Dx)^{-1}$.
But, unfortunately this operator is not good in unbounded domains
(in particular, does not map $L^2(\R^3)$ to $L^2(\R^3)$) and this
makes the greater part of the analytic tools developed for the CH
equation unapplicable to the case of unbounded domains. One exception is the case of cylindrical domains $\Omega$ and {\it Dirichlet} boundary conditions where the inverse Laplacian is well-defined and the theory of {\it finite-energy} solutions can
be built straightforwardly combining the usual Cahn-Hilliard
technique and the technique of weighted estimates (see
\cite{Ab2,BV,MZ,EZ,Z1}). The {\it infinite-energy} solutions (including regular and singular potentials, uniqueness, dissipativity and attractors, etc.) for that case have been recently studied in \cite{EKZ}, see also \cite{Bo}.
\par
 However,
despite the general theory of dissipative PDEs in unbounded domains
which seems highly developed now-a-days (see the surveys \cite{MZ}
and \cite{Ba} and references therein),  behavior of
solutions of the CH equation {\it in the whole space} remains badly understood. Indeed, to
the best of our knowledge, only the {\it local} results in this
direction are available in the literature, like the nonlinear
(diffusive) stability of relatively simple equilibria (e.g.,
kink-type solutions), relaxation rates to that equilibria,
asymptotic expansions in a small neighborhood of them, etc., see
\cite{BKT,TKT} and references therein (we also mention \cite{DKS} where some results on the long-time behavior of finite-energy solutions are obtained for the case of so-called viscous CH equations).
\par
 The situation becomes even
worse for more general infinite-energy solutions (e.g., for the
initial data belonging to $L^\infty(\R^n)$ only). In this case, even
the global
 existence of a solution was not known for the simplest cubic nonlinearity $f(u)=u^3-u$ (again, to the best of our knowledge) and the
 boundedness of  solutions as $t\to\infty$ is established only if $f(u)$ is {\it linear} outside of a compact set in $\R$, see \cite{CM}.
 \par
The aim of the present paper is exactly to give the positive answer on the question of global well-posedness of the CH equations in $\R^3$ in the class of uniformly local spaces. Although our estimates are not strong enough to establish the global boundedness
of solutions in $L^\infty(\R^3)$ like in \cite{CM}, we are able to treat a much more general class of nonlinearities including the case of regular potentials with arbitrary polynomial growth rate as well as the case of sufficiently strong singular potentials. Moreover, our upper bounds on solutions show that they can grow at most polynomially in time, see Theorems \ref{Th0.1} and \ref{Th.critical} below.
 \par
 In addition, we also study the so-called Cahn-Hilliard-Oono equation
\begin{equation}\label{0.ch}
\Dt u=\Dx(-\Dx u+f(u)+g)-\lambda u,
\end{equation}
where the extra term $\lambda u$, $\lambda>0$ models the non-local long-ranged interactions, see \cite{Oop, Miro} for more details.
As we will show below, the presence of this extra term drastically changes the situation with regard to the long-time behavior
 and allows us to obtain not only the global boundedness of solutions, but also their  {\it dissipativity}, and existence of a locally-compact global attractor, etc. Note that these results hold for all polynomial nonlinearities (satisfying the natural dissipativity assumptions)
 as well as for sufficiently strong singular potentials. Although we cannot cover the case of logarithmic nonlinearities,
  we are able to treat equation \eqref{0.ch} with
\begin{equation}\label{0.fpol}
f(u)\sim\frac{u}{(1-u^2)^\gamma}-Ku
\end{equation}
and $\gamma\ge5/3$.
\par
The paper is organized as follows.
\par
The definitions and key facts on the weighted and uniformly local Sobolev spaces which will be used throughout of the paper are briefly discussed in Section \ref{s1}.
\par
The key a priori estimate which gives at most polynomial growth rate in time for the infinite-energy solutions of \eqref{00.ch} is derived in Section \ref{s2} for the case of regular potentials. The uniqueness and further regularity of these solutions are verified in Section \ref{s3}.
\par
Section \ref{s4} is devoted to the extension of the above results to the Cahn-Hilliard-Oono equation. In particular, the dissipative estimate for the solutions of that equation is presented there.
\par
Finally, the Cahn-Hilliard and Cahn-Hilliard-Oono  equations with singular potentials are considered in Section \ref{s5}.

\section{Preliminaries: weighted and uniformly local spaces}\label{s1}
In this section, the definitions and key properties of the weighted and uniformly local Sobolev spaces needed for what follows
are briefly recorded, see \cite{EZ,MZ,Z3,Z4} for a more detailed exposition. We start with introducing the class of suitable weight functions.
\par
A function $\phi \in L^\infty_{loc}(\R^3)$ is called a weight function
with exponential rate of growth ($\nu>0$) if the conditions
\begin{equation}\label{2.weightexp}
\phi(x)>0\qand \phi(x + y) \leq C e^{\nu|x|} \phi(y)\,,
\end{equation}
are satisfied for every $x, y \in \R^3$.
\par
Any  weight function with growth rate $\nu$ also satisfies
$$
\phi(x + y) \geq C^{-1} e^{-\nu|x|} \phi(y)
$$
for all $x,y \in \R^3$.  Important
examples of weight functions with growth rate $\nu$, are
\begin{equation}
 \label{weight}
\phi_\eb(x) = \frac{1}{(1 + |\eb x|^2)^\frac{\gamma}{2}}\qand
\varphi_\eb(x)=e^{-\sqrt{|\eb x|^2+1}}\,,
\end{equation}
where $\gamma\in\R$ is arbitrary and $\eb < \nu$ in the second example. Crucial for what follows is the fact
that these functions satisfy \eqref{2.weightexp} uniformly with respect to $\eb\to0$. Moreover, if $\phi(x)$ satisfies \eqref{2.weightexp}, then the shifted weight function $\phi(x-x_0)$, $x_0\in\R^3$, also satisfies \eqref{2.weightexp} with the same constants $C$ and $\nu$.
\par
It is not difficult to check that the second weight function satisfies
\begin{equation}\label{1.weight}
|D^N_x\varphi_{\eb}(x)|\le C_N\eb^N\varphi_{\eb}(x)
\end{equation}
for all $N\in\Bbb N$ and the constant $C_N$ is independent of $\eb\to0$ (here and below $D^N_x$ stands for the collection of all partial derivatives of order $N$ with respect to $x$). In addition, the first weight function $\phi_\eb(x)$ satisfies the improved version of \eqref{1.weight}
\begin{equation}\label{0.6}
|D^N_x\phi_{\eb}(x)|\le C_N\eb^N [\phi_{\eb}(x)]^{1+N/\gamma}\le C_N'\eb^N\phi_{\eb}(x),
\end{equation}
where $C'_N$ is also independent of $\eb\to0$. Furthermore, to verify the dissipativity  of the Cahn-Hilliard-Oone equation, we will need to consider the weight functions $\phi_{\eb(t)}(x)$ where the parameter $\eb=\eb(t)$ depends explicitly on time. In this case,
\begin{equation*}
\partial_t\phi_{\eb(t)}(x)=\eb'(t)\frac{x}{\eb(t)}\cdot\nabla_x\phi_{\eb(t)}(x)=\gamma\frac{\eb'(t)}{\eb(t)}\phi_{\eb(t)}(x)\frac{|\eb(t)x|^2}{1+|\eb(t) x|^2}
\end{equation*}
and, therefore,
\begin{equation}\label{ebt}
|\partial_t\phi_{\eb(t)}(x)|\le \gamma\frac{|\eb'(t)|}{\eb(t)}\phi_{\eb(t)}(x),\ \ x\in\R^3.
\end{equation}
We are now ready to introduce the weighted and uniformly local spaces which will be used throughout the paper.
\begin{definition}\label{Def.spaces} For any $1\le p\le\infty$, the uniformly local Lebesgue space $L^p_b(\R^3)$ is defined as follows:
\begin{equation}\label{1.ul}
L^p_b(\R^3):=\left\{u\in L^p_{loc}(\R^3) : \|u\|_{L^p_b}:=\sup_{x_0\in\R^3}\|u\|_{L^p(B^1_{x_0})}<\infty\right\},
\end{equation}
 where $B^R_{x_0}$ stands for the $R$-ball in $\R^3$ centered at $x_0$.
\par
Furthermore, for any  weight function
$\phi(x)$  with exponential rate of growth, we define the weighted Lebesgue spaces:
\[
L^p_\phi(\R^3) = \left\{ u \in L^p_{loc}(\R^3) : \|u\|_{L^p_\phi}:=\(\int_{\R^3} \phi(x) |u(x)|^p dx\)^{1/p} < \infty\right\}.
\]
Analogously, the weighted  ($W^{l,p}_\phi(\R^3)$) and uniformly local ($W^{l,p}_b(\R^3)$)  are defined as subspaces of $\mathscr{D}'(\R^3)$ of distributions whose derivatives up to order $l$  belong to  $L^p_\phi(\R^3)$ or $L^p_b(\R^3)$ respectively (this works for natural $l$ only, the weighted Sobolev spaces with fractional/negative number of derivatives can be also defined in a standard way using interpolation/duality).
\end{definition}
\par
The following  proposition gives the technical tool for estimating the uniformly local norms of solutions using the energy estimates in weighted Sobolev spaces.

 \begin{proposition}\label{Prop.weighted} Let $\phi$ be a weight function with exponential growth such that $\|\phi\|_{L^1(\R^3)}<\infty$ and let $u\in L^p_b(\R)$ for some $1\le p<\infty$. Then, $u\in L^p_\phi(\R^3)$ and
 \begin{equation}\label{2.un-w}
 \|u\|_{L^p_\phi}\le C\|\phi\|_{L^1}^{1/p}\|u\|_{L^p_b}\,,
 \end{equation}
where the constant $C$ depends only on $p$ and the constants $C$ and $\nu$ in \eqref{2.weightexp} (and is independent of the concrete choice of the functions $\phi$ and $u$). Moreover,
\begin{equation}\label{2.w-un}
\|u\|_{L^2_b}\le C\sup_{x_0\in\R^3}\|u\|_{L^2_{\phi(\cdot-x_0)}}\,,
\end{equation}
where $C$ is also independent of the concrete choice of $u$ and $\phi$.
\end{proposition}
For the proof of this proposition, see \cite{EZ} or \cite{Z4}.
\par
We will mainly use estimate \eqref{2.un-w} in the situation where $\phi=\phi_\eb$ is one of the special weight functions of \eqref{weight} and $\eb>0$ is a small parameter. In this case, for $\gamma>3$, we have $\|\phi_\eb\|_{L^1}\sim \eb^{-3}$ and \eqref{2.un-w} reads
\begin{equation}\label{1.mainest}
\|v\|_{L^p_{\phi_\eb}}\le
C\eb^{-3/p}\|v\|_{L^p_b},
\end{equation}
where the constant $C>0$ is independent of $\eb\to0$.
\par
In the sequel, we will need also the proper spaces for functions of time with values in some uniformly local space, say $u:[0,T]\to W^{l,p}_b(\R^3)$.
With a slight abuse of notations, we denote by $L^q_b([0,T],W^{l,p}(\R^3))$ the subspace of distributions generated by the following norm:
\begin{equation}\label{1.stul}
\|u\|_{L^q_b([0,T],W^{l,p}_b)}:=\sup_{(t,x_0)\in[0,T-1]\times\R^3}\|u\|_{L^q([t,t+1],W^{l,p}(B^1_{x_0}))}.
\end{equation}
\begin{remark} Note that a more standard definition of $L^q_b([0,T],W^{l,p}_b(\R^3))$ would be via the following norm:
\begin{equation}\label{bad}
\|u\|_{L^q_b([0,T],W^{l,p}_b(\R^3))}:=\sup_{t\in[0,T-1]}\(\int_{t}^{t+1}\|u(t)\|^q_{W^{l,p}_b}\,dt\)^{1/q}
\end{equation}
which {\it differs} from \eqref{1.stul} by the changed order of supremum over $x_0\in\R^3$ and integral in time and is slightly {\it stronger} than \eqref{1.stul}. The main reason to use \eqref{1.stul} instead of \eqref{bad} is that the first norm can be estimated through the associated weighted space analogously to \eqref{2.w-un} which is essential since all estimates in uniformly local spaces are usually obtained with the help of the associated weighted estimates. Thus, exactly the first norm gives the natural and useful generalization of the space $L^q([0,T],W^{l,p}(\R^3))$ to the uniformly local case and the second norm \eqref{bad} which requires more delicate additional arguments to be properly estimated has only a restricted interest.
\end{remark}

\section{The key estimate and global existence}\label{s2}
In this section, we derive the key a priori weighted estimate for  the solutions of the  Cahn-Hilliard equation
\begin{equation}\label{CH}
\Dt u=\Dx\mu,\ \ \mu:=-\Dx u+f(u)+g(x),\ \ u\big|_{t=0}=u_0
\end{equation}
which allows to verify the existence of global in time solutions in the proper uniformly local Sobolev spaces (the uniqueness of that solutions will be discussed in the next section).  Here $u=u(t,x)$ and $\mu=\mu(t,x)$ are unknown order parameter and chemical potential respectively, $\Dx$ is a Laplacian with respect to $x$ and $f$ and $g$ are given nonlinearity and external forces respectively.
\par
We  assume that the nonlinearity $f(u)=f_0(u)+\psi(u)$ satisfies
\begin{equation}\label{0.1}
\begin{cases}
1.\ f'_0(u)\ge 1,\ \ f_0(0)=0,\\
2.\ |\psi(u)|+|\psi'(u)|\le C,\\
3.\ \ |f(u)|\le \alpha|F(u)|+C,\ \ F(u):=\int_0^uf(v)\,dv,
\end{cases}
\end{equation}
where $\alpha>0$. In particular, all polynomials of odd order with positive first coefficient satisfy these assumptions, in addition, some potentials of exponential growth rate are also allowed by these assumptions. Note that, due to the first and second assumptions of \eqref{0.1}, we have
$$
F(u)\ge \beta|u|^2-C,\ \ u\in\R
$$
for some positive constants $\beta$ and $C$.
\par
 We also assume that the external forces $g\in L^6_b(\R^3)$ and the initial data
$u_0$ belongs to the space $\Phi_b$ defined as follows:
\begin{equation}\label{0.2}
\Phi_b:=\{u\in W^{1,2}_b(\R^3),\ \ F(u)\in L^1_b(\R^3)\}.
\end{equation}
We say that $u$ is a solution of \eqref{CH} on the time interval $t\in[0,T]$ if
\begin{equation}\label{0.3}
\begin{cases}
1.\ \  u(t)\in\Phi_b,\ t\in[0,T];\\
2. \ \ u(t)\in C([0,T],L^2_{loc}(\R^3));\\
3.\ \ \mu\in L^2_b([0,T], W^{1,2}_b(\R^3))
\end{cases}
\end{equation}
and equation \eqref{CH} is satisfied in the sense of distributions. Note that the third assumption of \eqref{0.3}, together with the maximal regularity
theorem for the semilinear equation
$$
-\Dx u+f(u)=\mu-g
$$
in the uniformly local space $L^6_b(\R^3)$, imply that
$$
u\in L^2_b([0,T],W^{2,6}_b(\Omega)),\ \ f(u)\in L^2_b([0,T],L^6_b(\R^3)).
$$
The main result of this section is the following theorem.

\begin{theorem}\label{Th0.1} Let the above assumptions hold. Then, for every $T>0$, equation \eqref{CH} possesses at least one solution in the sense of \eqref{0.3} which satisfies the following estimate:
\begin{multline}\label{0.4}
\|u(t)\|_{W^{1,2}_b}^2+\|F(u(t))\|_{L^1_b}+\|\Nx\mu\|_{L^2_b([0,t]\times\R^3)}^2\le\\\le C(1+ t^4)\(1+\|g\|_{L^6_b}^2+ \|u_0\|_{W^{1,2}_b}^2+\|F(u_0)\|_{L^1_b}\)^{5/2},
\end{multline}
where the constant $C$ is independent of $u$, $g$ and $t$.
\end{theorem}
\begin{proof} We will give below only the formal derivation of the key a priori estimate \eqref{0.4}. The existence of solutions can be then deduced
using the approximation of the infinite energy data $(u_0,g)$ by the finite energy functions $(u_0^n,g^n)$ (for which the existence and regularity of
a solution is immediate and the derivation of \eqref{0.4} is justified) and passing to the limit in a local topology. Since these arguments are standard, see e.g., \cite{EZ,EKZ,Z1}, we leave them to the reader and concentrate ourselves on the derivation of the key estimate.
\par
Let us use the polynomial weight functions
\begin{equation}\label{0.5}
\phi(x):=\frac1{(1+|x|^2)^{5/2}},\ \ \phi_{\eb,x_0}(x):=\phi(\eb(x-x_0)),\ \eb>0,\ \ x_0\in\R^3.
\end{equation}
which corresponds to \eqref{weight} with $\gamma=5$ and satisfies the assumption $\phi_\eb\in L^1(\R^3)$. Then, these weights satisfy estimate \eqref{0.6} uniformly also with respect to $x_0\in\R^3$.
\par
We now multiply equation \eqref{CH} by $\phi_\eb\mu=\phi_{\eb,x_0}\mu$ and integrate over $x$. Then, after  straightforward computations, we get
\begin{multline}\label{0.7}
\frac d{dt}\((F(u),\phi_\eb)+\frac12(|\Nx u|^2,\phi_\eb)+(g,\phi_\eb u)\)+(|\Nx\mu|^2,\phi_\eb)=\\=\frac{1}{2}(|\mu|^2,\Dx\phi_\eb)+(\Nx\mu,\Nx(\Nx\phi_\eb\cdot\Nx u)).
\end{multline}
In order to estimate the left-hand side of this equation, we need the following two lemmas.
\begin{lemma}\label{Lem0.1} Let the above assumptions hold. Then the following inequality holds:
\begin{equation}\label{0.8}
(\phi_\eb^3,|f(u)|^6)\le  C_1\|\Nx(\phi_\eb^{1/2}\mu)\|_{L^2}^6+C_2\eb^{-3}(1+\|g\|^6_{L^6_b}),
\end{equation}
where the constants $C_i$ are independent of $\eb$.
\end{lemma}
\begin{proof} We rewrite the equation for $\mu$ as follows
\begin{equation}\label{0.mu}
-\Dx u+f_0(u)=\mu-\psi(u)-g
\end{equation}
and multiply it by $\phi_{\eb}^3f_0(u)|f_0(u)|^4$. Then, integrating by parts and using that $f_0'(u)\ge0$, we end up with
\begin{equation}\label{0.f06}
(\phi_\eb^3,|f_0(u)|^6)\le |(\Dx(\phi_\eb^3),F_6(u))|+(\phi^3_\eb(\mu+\psi(u)-g,f_0(u)|f_0(u)|^4),
\end{equation}
where $F_6(u):=\int_0^uf_0(u)|f_0(u)|^4\,du$. We also note that, due to the monotonicity of $f_0$,
$$
|F_6(u)|\le |u||f_0(u)|^5\le|f_0(u)|^6
$$
and, therefore, due to \eqref{0.6}, the first term in the right-hand side of \eqref{0.f06} is absorbed by the left-hand side if $\eb$ is small enough.
Then, estimating the second term in the right-hand side by H\"older inequality and using that $\psi$ is bounded, we obtain
\begin{equation}\label{0.f0b}
(\phi_\eb^3,|f_0(u)|^6)\le C(\phi_\eb^3,|\mu|^6)+C(|g|^6+1,\phi_\eb^3).
\end{equation}
Using now estimate \eqref{1.mainest} with $p=1$
  together with the Sobolev inequality
\begin{equation}\label{0.h1l6}
\|v\|_{L^6}\le C\|\Nx v\|_{L^2},
\end{equation}
we end up with
$$
(\phi_\eb^3,|f_0(u)|^6)\le C\|\Nx(\phi_\eb^{1/2}\mu)\|^6_{L^2}+C\eb^{-3}(\|g\|^6_{L^6_b}+1)
$$
which implies \eqref{0.8} and finishes the proof of the lemma.
\end{proof}
\begin{lemma}\label{Lem0.2} Let the above assumptions hold. Then, the following estimate is valid:
\begin{equation}\label{0.h2h1}
(\phi_\eb,|D^2_x u|^2)\le C(\phi_\eb,|\Nx \mu|^2)+C(\phi_\eb,|\Nx u|^2)+C(\phi_\eb,|g|^2),
\end{equation}
where the constant $C$ is independent of $\eb$ and $D^2_xu$ means the collection of all second derivatives of $u$ with respect to $x$.
\end{lemma}
\begin{proof}
Indeed, multiplying \eqref{0.mu} by $-\Nx(\phi_\eb\Nx u)$, integrating by parts and using the monotonicity of $f_0(u)$ together with \eqref{0.6}, we have after standard estimates that
\begin{equation}\label{0.dh1}
(\phi_\eb,|\Dx u|^2)\le C(\phi_\eb,|\Nx \mu|^2)+C(\phi_\eb,|\Nx u|^2)+C(\phi_\eb,|g|^2),
\end{equation}
so we only need to estimate the mixed second derivatives of $u$ based on \eqref{0.dh1}. To this end, we note that
$$
(\phi_\eb,|\Dx u|^2)=\sum_{i,j}(\phi_\eb\partial_{x_i}^2u,\partial_{x_j}^2u)
$$
and, for $i\ne j$
\begin{multline*}
(\phi_\eb\partial_{x_i}^2u,\partial_{x_j}^2u)=-(\phi_\eb\partial_{x_i}u,\partial_{x_i}\partial_{x_j}^2u)-
(\partial_{x_i}\phi_\eb\partial_{x_i}u,\partial_{x_j}^2u)=(\phi_\eb,|\partial^2_{x_ix_j}u|^2)+\\+
(\partial_{x_j}\phi_\eb,\partial_{x_i}u,\partial_{x_i}\partial_{x_j}u)+
(\partial_{x_i}\phi_\eb,\partial^2_{x_ix_j}u\partial_{x_j}u)+
(\partial^2_{x_ix_j}\phi_\eb,\partial_{x_i}u\partial_{x_j}u)=\\=
(\phi_\eb,|\partial_{x_ix_j}^2u|^2)+(\partial^2_{x_ix_j}\phi_\eb,\partial_{x_i}u\partial_{x_j}u)-
\frac12(\partial^2_{x_i}\phi_\eb,|\partial_{x_j}u|^2)-\frac12(\partial_{x_j}^2\phi_\eb,|\partial_{x_i}^2u|^2)
\end{multline*}
which together with \eqref{0.dh1} and \eqref{0.6} imply \eqref{0.h2h1} and finishes the proof of the lemma.
\end{proof}
We are now ready to finish the derivation of the key estimate. To this end, we note that due to assumptions \eqref{0.1} on the nonlinearity $f(u)$, there exists a constant $C$ (independent of $\eb$) such that the function
\begin{equation}\label{0.wen}
E_{\phi_\eb}(t):=(F(u(t)),\phi_\eb)+\frac12(|\Nx u(t)|^2,\phi_\eb)+C\eb^{-3}(\|g\|^2_{L^2_b}+1)
\end{equation}
satisfies the inequalities
\begin{multline}\label{0.weq}
(|F(u(t))|,\phi_\eb)+\frac12(|\Nx u(t)|,\phi_\eb)\le E_{\phi_\eb}(t)\le\\\le (|F(u(t))|,\phi_\eb)+\frac12(|\Nx u(t)|,\phi_\eb)+C'\eb^{-3}(1+\|g\|^2_{L^2_b}),
\end{multline}
where the constant $C'$ is independent of $\eb$. Then as
$$
 0 \leq - \| \nabla_x(\phi_\eb^{1/2} \mu) \|^2_{L^2} + \frac{1}{4}(\phi_\eb^{-1} (\nabla_x \phi_\eb)^2, |\mu|^2) + \frac{1}{2}(\phi_\eb, |\nabla_x \mu|^2),
 $$
we can add it to the right hand side of \eqref{0.7} and this reads
\begin{multline}\label{0.huge}
\frac{d}{dt} E_{\phi_\eb}(t) + \frac{1}{2}(|\nabla_x \mu|^2, \phi_\eb) + \frac{1}{2}\| \nabla_x(\phi_\eb^{\frac{1}{2}} \mu) \|^2_{L^2}\le
 C(\phi_\eb^{-1}(\nabla_x \phi_\eb)^2 + \Delta_x\phi_\eb, |\mu|^2) +\\+ \frac{1}{4}(\phi_\eb, |\nabla_x \mu|^2) + C(|D^2_x \phi_\eb|^2 \ |\nabla_x u|^2, \phi_\eb^{-1}) + (|D^2_x u|^2 \ |\nabla_x \phi_\eb|^2, \phi_\eb^{-1}).
\end{multline}
Expanding $\mu = -\Delta_x u + f(u) + g$ in the first term at the right-hand side and using \eqref{0.6} together with \eqref{1.mainest}, we obtain
\begin{multline}\label{0.huge1}
\frac{d}{dt} E_{\phi_\eb}(t) + \frac{1}{4}(|\nabla_x \mu|^2, \phi_\eb) + \frac{1}{2}\| \nabla_x(\phi_\eb^{\frac{1}{2}} \mu) \|^2_{L^2}
\le\\\le C(\epsilon^2 \phi_\eb^{7/5}, |\Delta_x u|^2 + |f(u)|^2 + |g|^2)+ C(\epsilon^2 \phi_\eb^{9/5}, \ |\nabla_x u|^2) + C(|D^2_x u|^2, \epsilon^2 \phi_\eb^{7/5})\le\\\le C\eb^2(\phi_\eb^{7/5},|f(u)|^2)+C\eb^2(\phi_\eb,|\Nx u|^2)+C\eb^{-1}\|g\|^2_{L^2_b}+C\eb^2(\phi_\eb,|D^2_xu|^2).
\end{multline}
Using Lemmas  \ref{Lem0.2} and \ref{Lem0.1} to estimate the last terms into the left and right-hand sides respectively, we arrive at
\begin{multline}\label{0.better}
\frac{d}{dt} E_{\phi_\eb}(t) + \beta\(|\nabla_x \mu|^2, \phi_\eb) + \| \nabla_x(\phi_\eb^{1/2} \mu) \|^2_{L^2} + (\phi_\eb^3, |f(u)|^6)^{\frac{1}{3}}\)\le\\
\le C\eb^2(\phi_\eb, |\nabla_x u|^2) + C\eb^2(\phi_\eb^{7/5},|f(u)|^2) + C\eb^{-3}(\|g\|^6_{L^6_b} + 1),
\end{multline}
where $\beta$ is some positive constant independent of $\eb\to0$.
\par
It only remains to estimate the second term in the right-hand side of \eqref{0.better}. To this end, we will interpolate between the $L^1_{\phi_\eb}$
 and $L_{\phi_\eb^3}^6$. Namely, due to the H\"older inequality,
\begin{multline}\label{0.interp}
 \eb^2(\phi_\eb^{\frac{7}{5}},|f(u)|^2)=\eb^2([\phi_\eb|f(u)|]^{4/5},[\phi_\eb^{1/2}|f(v)|]^{6/5})\le\\\le \eb^2(\phi_\eb,|f(u)|)^{4/5}(\phi_\eb^3,|f(u)|^6)^{1/5}\le
 C\eb^5(\phi_\eb,|f(u)|)^2+\beta(\phi_\eb^3,|f(u)|^6)^{1/3}
\end{multline}
and, therefore, using the third assumption of \eqref{0.1} and \eqref{0.weq}, we finally arrive at
\begin{equation}\label{0.good}
\frac d{dt}E_{\phi_\eb}(t)+\beta\|\Nx\mu\|_{L^2_{\phi_\eb}}^2\le C\eb^2 E_{\phi_\eb}(t)+C\eb^5[E_{\phi_\eb}(t)]^2+C\eb^{-3}(\|g\|_{L^6_b}^2+1).
\end{equation}
We claim that \eqref{0.good} is sufficient to derive the key estimate \eqref{0.4} and finish the proof of the theorem. Indeed, due to \eqref{1.mainest},
\begin{equation}\label{0.ubb}
E_{\phi_\eb}(0)\le C\eb^{-3}(\|u_0\|_{W^{1,2}_b}^2+\|F(u_0)\|_{L^1_b}+\|g\|_{L^2_b}^2+1)
\end{equation}
and, using in addition that $\eb^2 y\le \eb^5y^2+\eb^{-1}$, we see that the function $V_\eb(t):=\eb^{3}E_{\phi_\eb}(t)$ solves the inequality
\begin{equation}\label{0.last}
\frac d{dt}V_\eb+\beta\eb^{3}\|\Nx \mu\|^2_{L^2_{\phi_\eb}}\le \eb^2V_\eb^2+C(\|g\|_{L^6_b}^2+1),\ V_\eb(0)\le C(1+\|g\|_{L^6_b}^2+\|u_0\|_{\Phi_b}),
\end{equation}
where the constant $C$ is independent of $\eb$ and
\begin{equation}\label{phase}
\|u_0\|_{\Phi_b}:=\|u_0\|_{W^{1,2}_b}^2+\|F(u_0)\|_{L^1_b}.
\end{equation}
Let us now fix an arbitrary $T>0$ and consider inequality \eqref{0.last} on the time interval $t\in[0,T]$ only. Assume, in addition, that the $\eb>0$ is chosen in such way that the inequality
\begin{equation}\label{0.eb}
\eb^2V^2_\eb(t)\le C(1+\|g\|^2_{L^6_b}+\|u_0\|_{\Phi_b})
\end{equation}
is satisfied for all $t\in[0,T]$. Then, from \eqref{0.last}, we conclude that
\begin{equation}\label{0.est}
V_\eb(t)\le 2C(T+1)(1+\|g\|^2_{L^6_b}+\|u_0\|_{\Phi_b}),\ \ t\in[0,T].
\end{equation}
Thus, in order to satisfy \eqref{0.eb}, we need to fix $\eb=\eb(T,u_0,g)$ as follows
\begin{equation}\label{0.ebb}
\eb:=\frac1{2(T+1)[C(1+\|g\|^2_{L^6_b}+\|u_0\|_{\Phi_b})]^{1/2}}.
\end{equation}
Then, as it is not difficult to see, inequality \eqref{0.est} will be indeed satisfied for all $t\in[0,T]$ (with $\eb$ fixed by \eqref{0.ebb}) which gives
\begin{equation}\label{0.eest}
E_{\phi_\eb}(T)\le \eb^{-3}V_\eb(T)\le C(T+1)^4(1+\|g\|^2_{L^6_b}+\|u_0\|_{\Phi_b})^{5/2}.
\end{equation}
We now recall that $\phi_\eb(x)=\phi_{\eb,x_0}$ depends on the parameter $x_0\in\R^3$ and estimate \eqref{0.eest} is {\it uniform} with respect to this parameter. Thus, taking the supremum with respect to this parameter and using \eqref{2.w-un}, we finally have
\begin{equation}\label{0.b}
\|u(T)\|^2_{W^{1,2}_b}+\|F(u(T))\|_{L^1_b}\le 2\sup_{x_0\in\R^3}E_{\phi_{\eb,x_0}}(T)\le C(T+1)^4(1+\|g\|^2_{L^6_b}+\|u_0\|_{\Phi_b})^{5/2}.
\end{equation}
Estimate \eqref{0.b} together with \eqref{0.last} (which we need for estimating the gradient of $\mu$) imply \eqref{0.4} and finishes the proof of the theorem.
\end{proof}
\begin{corollary} Under the assumptions of Theorem \ref{Th0.1}, the solution $u$ possesses the following additional regularity:
\begin{multline}\label{0.reg}
\|\Dt u\|_{L^2([0,T],W^{-1,2}(B^1_{x_0}))}^2+\|\mu\|_{L^2([0,T],L^6(B^1_{x_0}))}^2+
\|f(u)\|_{L^2([0,T],L^6(B^1_{x_0}))}^2+\\+\|u\|_{L^2([0,T],W^{2,6}(B^1_{x_0}))}^2\le C(T+1)^4(1+\|g\|^2_{L^6_b}+\|u_0\|_{\Phi_b})^{5/2},
\end{multline}
where the constant $C$ is independent of $x_0\in\R^3$, $T$ and $u_0$.
\end{corollary}
Indeed, the estimate of the first term in the left-hand side follows from the identity $\Dt u=\Dx\mu$ and estimate \eqref{0.4},
the estimate for the $L^6$-norm of $\mu$ follows from the presence of the term $\|\Nx(\phi_\eb^{1/2}\mu)\|^2_{L^2}$ in the left-hand side
of \eqref{0.huge} and Sobolev embedding. Finally, the estimate for  two last terms in the left-hand side is a corollary of the corresponding
 estimate for $\mu$ and the $L^6$-maximal regularity for the semilinear equation \eqref{0.mu}.

\section{Uniqueness and further regularity}\label{s3}
The aim of this section is to verify that  the solution $u$ of the Cahn-Hilliard equation \eqref{CH} constructed in Theorem \ref{Th0.1} is unique
and to check that this solution is actually smooth. To this end, we need more assumptions on the nonlinearity $f$. Namely, we assume that there exists a {\it convex} positive function $\Psi$ such that
\begin{equation}\label{1.psi}
\begin{cases}
1.\ \  \Psi(u)\le C(|F(u)|+1),\ \\
2.\ \  |f'(u)|\le \Psi(u).
\end{cases}
\end{equation}
Note that all of the conditions \eqref{0.1} and \eqref{1.psi} on the nonlinearity $f$ do not look restrictive and are satisfied, e.g., for any polynomial of  odd order and positive highest  coefficient and even for some exponentially growing potentials.
\par
 The main result of this section is the following theorem.
\begin{theorem}\label{Th1.unique} Let the above assumptions hold and let $u_1,u_2$ be two solutions of problem \eqref{CH} satisfying \eqref{0.3}.
 Then, the following estimate holds:
\begin{equation}\label{1.2}
\|u_1(t)-u_2(t)\|_{W^{-1,2}_b}\le C_T\|u_1(0)-u_2(0)\|_{W^{-1,2}_b},
\end{equation}
where the constant $C_T$ depends only on $T$ and on the \eqref{0.3}-norms of the solutions $u_1$ and~$u_2$.
\end{theorem}
\begin{proof} Let $v(t):=u_1(t)-u_2(t)$. Then, this function solves the equation
\begin{equation}\label{1.3}
\Dt v=\Dx(-\Dx v+l(t)v),\ \ v\big|_{t=0}=v_0,\ l(t):=\int_0^1f'(su_1+(1-s)u_2)\,ds
\end{equation}
which we rewrite in the following equivalent form (adapted to the $H^{-1}$-energy estimates):
\begin{equation}\label{1.3.5}
(-\Dx+1)^{-1}\Dt v=\Dx v-l(t)v+(-\Dx+1)^{-1}(\Dx v-l(t)v).
\end{equation}
Let now $\varphi(x):=e^{-\sqrt{|x|^2+1}}$ be the exponential weight function (see \eqref{weight})  and let $\varphi_{\eb}(x)=\varphi_{\eb,x_0}(x)=\varphi(\eb(x-x_0))$. Then this function satisfies estimate \eqref{1.weight} uniformly also with respect to $x_0\in\R^3$.
\par
We multiply \eqref{1.3}
by the following function:
$$
-\Nx(\varphi_{\eb}\Nx((-\Dx+1)^{-1}v))+\varphi_{\eb}(-\Dx+1)^{-1}v=\varphi_{\eb} v-\Nx\varphi_{\eb}\cdot\Nx((-\Dx+1)^{-1}v)
$$
and integrate over $x$. Then, denoting by $w:=(-\Dx+1)^{-1}v$, we have
\begin{multline}\label{1.4}
\frac12\frac d{dt}(|\Nx w|^2+|w|^2,\varphi_{x_0})+(|\Nx v|^2,\varphi_{\eb})+(l(t)v,\varphi_{\eb} v)=(-\Dx\varphi_{\eb},|v|^2)-\\-
(\Dx v,\Nx\varphi_{\eb}\cdot\Nx w)+(l(t)v,\Nx\varphi_{\eb}\cdot\Nx w)+\\+
((-\Dx+1)^{-1}(\Dx v-l(t)v),\varphi_{\eb}v)
-((-\Dx+1)^{-1}(\Dx v-l(t)v),\Nx\varphi_{\eb}\cdot\Nx w).
\end{multline}
Using now the weighted maximal regularity for the equation $-\Dx w+w=v$, we have that, for sufficiently small $\eb$,
\begin{equation}\label{1.max}
\|w\|_{W^{s+2,2}_{\varphi_\eb}}\sim\|v\|_{W^{s,2}_{\varphi_\eb}},\ \ s\in\R,
\end{equation}
where the equivalence constants depend only on $s$, see \cite{EZ,MZ}. Therefore, integrating by parts and using \eqref{1.weight} together with \eqref{1.max} and Cauchy-Schwartz inequality, we end up with
\begin{multline}\label{1.big}
\frac d{dt}\|w\|^2_{W^{1,2}_{\varphi_\eb}}+\|v\|^2_{W^{1,2}_{\varphi_\eb}}+(|l(t)|v,\varphi_{\eb} v)\le C\|v\|^2_{L^2_{\varphi_\eb}}+\\+
(|l(t)|\cdot|v|,\varphi_\eb|\Nx w|+|(-\Dx+1)^{-1}(\varphi_\eb v)|+|(-\Dx+1)^{-1}(\Nx\varphi_\eb\cdot\Nx w)|).
\end{multline}
Let now
$$
h:=\varphi_\eb^{-1}(\varphi_\eb|\Nx w|+|(-\Dx+1)^{-1}(\varphi_\eb v)|+|(-\Dx+1)^{-1}(\Nx\varphi_\eb\cdot\Nx w)|).
$$
 Then, applying the Cauchy-Schwartz
inequality together with the weighted maximal regularity for the Laplacian, we conclude that
\begin{multline}\label{1.nest}
(|l(t)|v|,h)\le (|l(t)|,\varphi_\eb v^2)+(|l(t)|,\varphi_\eb h^2)\le (|l(t)|v,\varphi_\eb v)+\|l(t)\|_{L^1_b}\|\varphi_{\eb}^{1/2}h\|_{L^\infty}^2\le\\\le
(|l(t)|v,\varphi_\eb v)+C\|l(t)\|_{L^1_b}\|h\|_{W^{7/4,2}_{\varphi_\eb}}^2\le (|l(t)|v,\varphi_\eb v)+C\|l(t)\|_{L^1_b}\|v\|^2_{W^{3/4,2}_{\varphi_\eb}}.
\end{multline}
We now estimate the $L^1_b$-norm of $l(t)$ using assumptions \eqref{1.psi}. Namely,
\begin{multline}\label{1.lest}
\|l(t)\|_{L^1_b}\le \int_0^1\|f'(su_1+(1-s)u_2)\|_{L^1_b}\,ds\le\\\le\int_0^1\|\Psi(su_1+(1-s)u_2)\|_{L^1_b}\,ds\le \int_0^1\|s\Psi(u_1)+(1-s)\Psi(u_2)\|_{L^1_b}\,ds\le\\\le \|\Psi(u_1)\|_{L^1_b}+\|\Psi(u_2)\|_{L^1_b}\le C(\|F(u_1)\|_{L^1_b}+\|F(u_2)\|_{L^1_b}+1).
\end{multline}
Thus, $\|l(t)\|_{L^1_b}\le C_T=C_T(u_1,u_2)$ and using \eqref{1.nest}, we rewrite \eqref{1.big} as follows
\begin{equation}\label{1.small}
\frac d{dt}\|w\|^2_{W^{1,2}_{\varphi_\eb}}+\|v\|^2_{W^{1,2}_{\varphi_\eb}}\le C_T\|v\|^2_{W^{3/4,2}_{\varphi_\eb}}.
\end{equation}
Interpolating now the $W^{3/4,2}$-norm between the $W^{1,2}$ and $W^{-1,2}$-norms, we finally arrive at
\begin{equation}\label{1.fin}
\frac d{dt}\|w\|^2_{W^{1,2}_{\varphi_{\eb,x_0}}}\le C_T\|w\|^2_{W^{1,2}_{\varphi_{\eb,x_0}}},
\end{equation}
where the constant $C_T$ grows polynomially in $T$ (according to \eqref{0.4}). Applying the Gronwall inequality to this relation, taking the supremum over $x_0\in\R^3$ and using \eqref{2.w-un}, we end up with \eqref{1.2} and finish the proof of the theorem.
\end{proof}
The next several simple corollaries of the proved theorem show that the constructed solution $u$ is smooth.
\begin{corollary}\label{Cor1.dt} Let the above assumptions hold and let, in addition, $u_0$ be smooth enough to guarantee that $\Dt u(0)\in W^{-1,2}_b(\R^3)$. Then, $\Dt u(t)\in W^{-1,2}_b$ for all $t\ge0$ and its norm grows at most polynomially in time. If $\Dt u(0)\notin W^{-1,2}_b(\R^3)$, then nevertheless $\Dt u(t)\in W^{-1,2}_b(\R^3)$ for all $t>0$ and the following estimate holds:
\begin{equation}\label{1.dtsmooth}
\|\Dt u(t)\|_{W^{-1,2}_b}^2\le Ct^{-1}(1+t^N)Q(1+\|g\|^2_{L^6_b}+\|u_0\|_{\Phi_b})
\end{equation}
for some monotone function $Q$ and constants $C$ and $N$ independent of $u_0$ and $t$.
\end{corollary}
\begin{proof} Indeed, differentiating equation \eqref{CH} in time and denoting $v(t):=\Dt u(t)$, we see that $v$ solves the equation
\begin{equation}\label{1.diff}
\Dt v=\Dx(-\Dx v+f'(u)v),\ \ v\big|_{t=0}=\Dt u(0)
\end{equation}
which is almost identical to equation \eqref{1.3}. Therefore, denoting by $w(t):=(-\Dx+1)^{-1}v(t)$ and arguing exactly as in the proof of the theorem, we derive that
\begin{equation}\label{1.dteq}
\frac d{dt}\|\Dt u(t)\|^2_{W^{-1,2}_{\varphi_\eb}}\le C(t)\|\Dt u(t)\|^2_{W^{-1,2}_{\varphi_\eb}},
\end{equation}
where $C(t)$ grows polynomially in time. This estimate, together with \eqref{0.reg} proves both assertions of the corollary.
\end{proof}
\begin{corollary}\label{Cor1.w26} Let the above assumptions hold and let, in addition, $u_0\in W^{2,6}_b(\R^3)$. Then, $u(t)\in W^{2,6}_b(\R^3)$ for all $t\ge0$ and its norm grows at most polynomially in time. If $u(0)\notin W^{2,6}_b(\R^3)$, then nevertheless $u(t)\in W^{2,6}_b(\R^3)$ for all $t>0$.
\end{corollary}
\begin{proof} Rewriting the Cahn-Hilliard equation in the form
$$
\mu=-(-\Dx+1)^{-1}\Dt u(t)+(-\Dx+1)^{-1}\mu
$$
and using the maximal regularity for the Laplacian in the uniformly local spaces, we see that
\begin{equation}\label{1.mureg}
\|\mu(t)\|_{W^{1,2}_b}\le C\|\Dt u(t)\|_{W^{-1,2}_b}+C\|\mu(t)\|_{W^{-1,2}_b}.
\end{equation}
Moreover, according to Theorem \ref{Th0.1}, we have
\begin{multline}\label{1.muweak}
\|\mu(t)\|_{W^{-2,2}_b}\le C\|(-\Dx+1)^{-1}(\Dx u(t)-f(u(t))+g)\|_{L^2_b}\le\\\le
 C(\|u(t)\|_{W^{1,2}_b}+\|g\|_{L^2_b}+\|f(u(t))\|_{L^1_b})\le C(\|u(t)\|_{\Phi_b}+\|g\|_{L^6_b})\le C(t,u_0,g).
\end{multline}
Thus, due to \eqref{1.mureg}, \eqref{1.muweak} and interpolation, we have
\begin{multline}\label{1.mu06}
\|\mu(t)\|_{L^6_b}\le C\|\mu(t)\|_{W^{1,2}_b}\le \\\le C(\|\Dt u(t)\|_{W^{-1,2}_b}+\|u(t)\|_{\Phi_b}+\|g\|_{L^6_b})\le C\|\Dt u(t)\|_{W^{-1,2}_b}+C(t,u_0,g),
\end{multline}
where the constant $C(t,u_0,g)$ grows polynomially in time.
\par
Estimate \eqref{1.mu06} together with Corollary \ref{Cor1.dt} give the assertion of Corollary \ref{Cor1.w26} for the $L^6_b$-norm of $\mu$. In order to obtain the analogous assertions for the $W^{2,6}_b$-norm of $u$, we apply the $L^6_b$-maximal regularity theorem for the semilinear equation \eqref{0.mu} which gives
\begin{equation}\label{1.maxsem}
\|u(t)\|_{W^{2,6}_b}\le C(1+\|\mu(t)\|_{L^6_b})\le C_1(\|\Dt u(t)\|_{W^{-1,2}_b}+\|u(t)\|_{\Phi_b}+\|g\|_{L^6_b}),
\end{equation}
Thus, the corollary is proved.
\end{proof}
\begin{remark} \label{Rem1. nice} The proved regularity is more than enough to initialize  the standard bootstraping process and to verify that the factual smoothness of $u(t)$ is restricted by the smoothness of $f$ and $g$ only. In particular, if both of them are $C^\infty$-smooth the solution will be $C^\infty$-smooth as well. If they are, in addition, real analytic, one has the real analytic in $x$ solution $u(t,x)$ as well (for $t>0$).
\end{remark}

\section{Dissipative estimate for the Cahn-Hilliard-Oono equation}\label{s4}

In this section we apply the above developed techniques to the so-called  Cahn-Hillard-Oono equation
\begin{equation}\label{CHd}
 \Dt u=\Dx\mu - \lambda u,\ \ \mu:=-\Dx u+f(u)+g(x),\ \ u\big|_{t=0}=u_0
\end{equation}
which differs from the classical Cahn-Hilliard equation by the presence of an extra term $\lambda u$ where the constant $\lambda>0$. The extra dissipative term  has been initially introduced to model the long-range nonlocal interactions (see \cite{Oop} and also \cite{Miro} for further details) and essentially simplifies the analysis of the long-time behavior of the Cahn-Hilliard equations in unbounded domains and   as we will see,  guarantees the dissipativity of the equation in the uniformly local spaces. To be more precise, the following theorem can be considered as the main result of the section.

\begin{theorem}\label{Th0.1d} Let the assumptions of Theorem \ref{Th0.1} hold.
Then the Cahn-Hilliard-Oono equation \eqref{CHd} possesses at least one global in time solution in the sense of \eqref{0.3} (for all $T>0$) which satisfies the following estimate:
\begin{multline}\label{thd.1}
\|u(t)\|_{W^{1,2}_b}^2+\|F(u(t))\|_{L^1_b}+\|\Nx\mu\|_{L^2_b([t, t+1]\times\R^3)}^2\le\\\le
Q(\|g\|_{L^6_b}) + Q(\|u(0)\|^2_{W^{1, 2}_b} + \|F(u(0))\|_{L^1_b})e^{-\sigma t},\ \ t\ge0,
\end{multline}
for some monotone increasing function $Q$ and positive constant $\sigma$  independent of the initial data $u_0$ and $t\ge0$.
\end{theorem}
\begin{proof} The proof of this theorem is analogous to Theorem \ref{Th0.1d}, but the presence of the dissipative term $\lambda u$ produces the extra term $\lambda_0 V_\eb(t)$ in the left-hand side of \eqref{0.last} with some positive $\lambda_0$ independent of $\eb$ and this gives global existence and the dissipative estimate for $V_\eb(t)$ if $\eb=\eb(u_0,g)$ is small enough. However, this is still not enough to deduce the dissipative estimate for $u(t)$ since the parameter $\eb$ still depends on the initial data $u_0$. To overcome this difficulty,
we will consider the  {\it time-dependent} parameter $\eb=\eb(t)$. To be more precise, let $\phi_{\eb,x_0}(x)$ be the same as in \eqref{0.5}. Then, due to \eqref{ebt}, we have

\begin{equation}\label{0.7d}
 |\partial_t \phi_{\eb,x_0}(x)|\le C_t [\phi_{\eb,x_0}(x)], \ \ C_t:=5\cdot\frac{|\eb'(t)|}{\eb(t)}.
\end{equation}
We multiply equation \eqref{CHd} by $\phi_\eb \mu = \phi_{\eb(t), x_0} \mu$, where the function $\eb(t)$ will be specified below, and integrate over $x$. Then, analogously to \eqref{0.7}, we get
\begin{multline}\label{0.7dd}
\frac d{dt}\((F(u),\phi_\eb)+\frac12(|\Nx u|^2,\phi_\eb)+(g,\phi_\eb u)\)+(|\Nx\mu|^2,\phi_\eb)=\\=\frac{1}{2}(|\mu|^2,\Dx\phi_\eb)+(\Nx\mu,\Nx(\Nx\phi_\eb\cdot\Nx u))
 -\\-\lambda(\phi_\eb, |\nabla u|^2+f(u)u+gu)  + (\Dt \phi_\eb, F(u)+\frac12|\nabla u|^2+gu),
\end{multline}
where the extra two terms in the right-hand side are due to the extra term $\lambda u$ and the dependence of $\phi_\eb$ on time. Note that the assumptions \eqref{0.3} imply that
$$
F(u)\le f(u)u+C
$$
for some constant $C$ and, therefore,
\begin{equation*}
-\lambda(\phi_\eb, |\nabla u|^2+f(u)u+gu)\le -\lambda E_{\phi_\eb}(t)+C\eb^{-3}(1+\|g\|^2_{L^6_b}),
\end{equation*}
where $E_{\phi_\eb}(t)$ is defined by \eqref{0.wen}. Furthermore, using \eqref{0.7d} together with \eqref{0.weq}, we have
$$
(\Dt \phi_\eb, F(u)+\frac12|\nabla u|^2+gu)\le 2C_t E_{\phi_\eb}+C(C_t+1)\eb^{-3}(1+\|g\|^2_{L^6_b}).
$$
Thus, the extra terms are estimated as follows
\begin{equation}\label{extra}
-\lambda(\phi_\eb, |\nabla u|^2+f(u)u+gu)  + (\Dt \phi_\eb, F(u)+\frac12|\nabla u|^2+gu)\le -\lambda/2 E_{\phi_\eb}(t)+C\eb^{-1}(1+\|g\|^2_{L^6_b})
\end{equation}
if the parameter $\eb(t)$ satisfies the following extra condition:
\begin{equation}\label{ebsprime}
C_t=5\frac{|\eb'(t)|}{\eb(t)}\le \frac\lambda2
\end{equation}
which is assumed to be satisfied from now on. The rest of the terms in \eqref{0.7dd} can be estimated using Lemmas \ref{Lem0.1} and \ref{Lem0.2}
exactly as in the proof of Theorem \ref{Th0.1} which gives the following dissipative analogue of \eqref{0.good}:
\begin{equation}\label{0.goodd}
\frac d{dt}E_{\phi_\eb}(t)+\frac\lambda2 E_{\phi_\eb}(t)+\beta\|\Nx\mu\|_{L^2_{\phi_\eb}}^2\le C\eb^2 E_{\phi_\eb}(t)+C\eb^5[E_{\phi_\eb}(t)]^2+C\eb^{-3}(\|g\|_{L^6_b}^2+1).
\end{equation}
Leaving \eqref{0.ubb} unchanged and again using that $\eb^2 y\le \eb^5y^2+\eb^{-1}$, we see that the function $V_\eb(t):=\eb^{3}E_{\phi_\eb}(t)$ solves the dissipative analogue of  inequality \eqref{0.last}:
\begin{equation}\label{0.lastd}
\frac d{dt}V_\eb+ \frac\lambda2 V_\eb + \beta\eb^{3}\|\Nx \mu\|^2_{L^2_{\phi_\eb}}\le \eb^2V_\eb^2+C(\|g\|_{L^6_b}^2+1),\ V_\eb(0)\le C(1+\|g\|_{L^6_b}^2+\|u_0\|_{\Phi_b}):=V_0,
\end{equation}
where the constant $C$ is independent of $\eb$ and $\|u_0\|_{\Phi_b}$ is defined by \eqref{phase}.
\par
We claim that inequality \eqref{0.lastd} is enough to deduce the desired dissipative estimate \eqref{thd.1} and finish the proof of the theorem. Indeed,
restyling it as
\begin{equation}\label{0.lastdd}
 \frac d{dt}V_\eb+ \frac{\lambda}{4} V_\eb \le V_\eb \left(\eb^2V_\eb - \frac{\lambda}{4}\right) +C_g,
\end{equation}
where $C_g = C(1+\|g\|_{L^6_b}^2)$, we see that, under the assumption
\begin{equation}\label{0.diss}
 \eb^2(t)V_\eb(t) \leq \frac{\lambda}{4},\ \ t\ge0,
\end{equation}
the first term in the right-hand side of \eqref{0.lastdd} will be negative (and, therefore, can be omitted) and
$V_\eb(t)$ will satisfy the estimate
\begin{equation}\label{0.niced}
 V_\eb(t) \leq \frac{4 C_g}{\lambda} + V_0 e^{\frac{-\lambda}{4}t},\ \ t\ge0.
\end{equation}
Using this observation, it is not difficult to show that both estimates \eqref{0.niced} and \eqref{0.diss}
 will be satisfied if the parameter $\eb(t)$ is chosen in such way that
\begin{equation}\label{epsdef.1}
\eb^2(t)\left( \frac{4C_g}{\lambda} + V_0e^{\frac{-\lambda}{4}t}  \right) \le \frac{\lambda}{4}.
\end{equation}
Thus, we only need to fix the function $\eb(t)\ll1$ satisfying the two inequalities  \eqref{ebsprime} and \eqref{epsdef.1}. In
particular, we may take
\begin{equation}\label{epsdef.2}
 \eb(t) = \eb_0\(\frac{\lambda/4}{\frac{4C_g}{\lambda} + V_0e^{-\sigma t}}\)^{\frac12},
\end{equation}
where $\eb_0>0$ and $\sigma>0$ are proper small constants. Indeed, condition \eqref{epsdef.1}
will be satisfied if $\eb\le1$ and $\sigma\le \lambda/4$. In order to check \eqref{ebsprime}, we note that
\begin{equation}\label{epsdef.3}
\frac{|\eb'(t)|}{\eb(t)}=|\frac d{dt}\log\eb(t)|=\frac12\cdot \frac{V_0\sigma e^{-\sigma t}}{\frac{4C_g}{\lambda} + V_0e^{-\sigma t}}\le \frac12\sigma
\end{equation}
and \eqref{ebsprime} will also be satisfied if $\sigma\le \lambda/5$. Thus, for that choice of $\eb(t)$ estimate
\eqref{0.niced} is satisfied and, therefore,
\begin{equation}\label{0.eestd}
E_{\phi_{\eb(t),x_0}}(t)\le \eb(t)^{-3}V_\eb(t)\le
C\(C_g + V_0e^{-\sigma t}\)^{\frac32}
\(C_g + V_0 e^{-\sigma t}\)
\end{equation}
uniformly with respect to $x_0\in\R^3$. Taking the supremum with respect to $x_0\in\R^3$ from both sides of \eqref{0.eestd} and using \eqref{2.w-un}, we finally arrive at
\begin{equation}\label{0.bd}
\|u(t)\|^2_{W^{1,2}_b}+\|F(u(t))\|_{L^1_b}\le
Q(\|g\|_{L^6_b}) + Q(\|u(0)\|^2_{W^{1, 2}_b} + \|F(u(0))\|_{L^1_b})e^{-\sigma t}
\end{equation}
for the properly chosen monotone function $Q$ and positive constant $\sigma$.
Estimate \eqref{0.bd} together with \eqref{0.lastd}
(which we need for estimating the gradient of $\mu$) implies \eqref{thd.1} and finishes the proof of the theorem.
\end{proof}
Let us now discuss the uniqueness and further regularity of solutions for the case of the Cahn-Hilliard-Oono equation.

\begin{proposition}\label{Propd.unique} Let the assumptions of Theorem \ref{Th0.1d} hold and, in addition, \eqref{1.psi} be satisfied. Then the solution $u(t)$ constructed in Theorem \ref{Th0.1d} is unique and, for every two solutions $u_1(t)$ and $u_2(t)$ of the Cahn-Hilliard-Oono equation, estimate \eqref{1.2} holds.
\end{proposition}
Indeed, the presence of the extra term $\lambda u$ in \eqref{CHd} does not make any essential difference for the uniqueness proof which repeats almost word by word the proof of Theorem \ref{Th1.unique} and by this reason is omitted.
\par
The following corollary is the dissipative analogue of Corollary \ref{Cor1.dt}.

\begin{corollary}\label{Cor2.dt} Let the assumptions of Proposition \ref{Propd.unique} hold and let, in addition, the initial data $u_0$ be such that $\Dt u(0)\in W^{-1,2}_b(\R^3)$. Then, $\Dt u(t)\in W^{-1,2}_b(\R^3)$ for all $t>0$ and the analogue of dissipative estimate \eqref{thd.1} is valid:
\begin{equation}\label{dtdis}
\|\Dt u(t)\|_{W^{-1,2}_b}\le Q(\|\Dt u(0)\|_{W^{-1,2}_b}+\|u(0)\|_{\Phi_b})e^{-\gamma t}+Q(\|g\|_{L^6_b})
\end{equation}
for  proper monotone function $Q$ and positive constant $\gamma$. Moreover, if $\Dt u(0)\notin W^{-1,2}_b(\R^3)$ then, nevertheless, $\Dt u(t)\in W^{-1,2}_b(\R^3)$ and the following estimate holds:
\begin{equation}\label{3.dtsm}
\|\Dt u(t)\|_{W^{-1,2}_b}\le C t^{-1/2}Q(\|u(0)\|_{\Phi_b}+\|g\|_{L^6_b}),\ \ t\in(0,1]
\end{equation}
for some monotone increasing function $Q$ and positive $C$.
\end{corollary}
\begin{proof} Indeed, arguing exactly as in Corollary \ref{Cor1.dt}, we end up with estimate
\begin{equation}\label{3.dtgood}
\frac d{dt}\|\Dt u(t)\|_{\varphi_\eb}^2\le Q(\|u(t)\|_{\Phi_b})\|\Dt u(t)\|^2_{L^2_{\varphi_\eb}},
\end{equation}
where $Q(z)=C(1+z)^8$, see \eqref{1.lest} and \eqref{1.small}. Multiplying this inequality by $t$ and integrating in time, we get
$$
t\|\Dt u(t)\|^2_{W^{-1,2}_{\varphi_{\eb,x_0}}}\le (Q(\|u(t)\|_{\Phi_b})+1)\int_0^t\|\Nx\mu(t)\|_{L^2_{\varphi_{\eb,x_0}}}^2\,dt,\ \  t\in(0,1].
$$
Taking the supremum over all shifts $x_0\in\R^3$ and using \eqref{2.w-un} together with \eqref{thd.1}, we have
$$
\|\Dt u(t)\|^2_{W^{-1,2}_b}\le t^{-1}Q(\|u(0)\|_{\Phi_b}+\|g\|_{L^6_b})
$$
for some new monotone function $Q$. Thus, \eqref{3.dtsm} is verified. In addition, the last estimate gives that
$$
\|\Dt u(t+1)\|_{W^{-1,2}_b}\le Q(\|u(t)\|_{W^{-1,2}_b}+\|g\|_{L^6_b})
$$
which together with \eqref{thd.1} proves also the dissipative estimate \eqref{dtdis} for $t\ge1$. Finally, estimate \eqref{dtdis} on the finite time interval $t\in[0,1]$ follows directly from the Gronwall inequality applied to \eqref{3.dtgood} and the corollary is proved.
\end{proof}
Furthermore, the dissipative analogue of Corollary \ref{Cor1.w26} also holds.

\begin{corollary}\label{Cor.w26d} Let the assumptions of Proposition \ref{Propd.unique} hold and let, in addition $u_0\in W^{2,6}_b(\R^3)$. Then, $u(t)\in W^{2,6}_b(\R^3)$ for all $t>0$ and the analogue of \eqref{dtdis} holds. If $u_0\notin W^{2,6}_b(\R^3)$ then, nevertheless, $u(t)\in W^{2,6}_b(\R^3)$ for all $t>0$ and the analog of smoothing property \eqref{3.dtsm} also holds.
\end{corollary}

\begin{remark}\label{Rem.Oono-dis} As we have noted in Remark \ref{Rem1. nice}, the verified $W^{2,6}$-regularity of solutions
 allows us to obtain further smoothness of solutions (restricted only by the regularity of $g$ and $f$) by standard bootstrapping arguments. In the case of the Cahn-Hilliard-Oono equation, the obtained estimates for the higher norms will be also dissipative.
 \par
 Note also that the proved dissipative estimate \eqref{thd.1} together with the smoothing properties established in Corollaries \ref{Cor2.dt} and \ref{Cor.w26d} allow us to define the dissipative solution semigroup in the phase space $\Phi_b$
 \begin{equation}\label{2.semigroup}
 S(t): \Phi_b\to\Phi_b,\ \ S(t)u_0=u(t),\ \ \Phi_b:=\{u_0\in W^{1,2}_b(\R^3),\ F(u_0)\in L^1_b(\R^3)\}
 \end{equation}
and verify that this semigroup possesses and absorbing set bounded in $W^{2,6}_b(\R^3)$. This, together with the Lipschitz continuity \eqref{1.2} allows us, in turn, to establish the existence of the so-called locally compact global attractor $\mathcal A\subset W^{2,6}_b(\R^3)$
(see \cite{MZ} for more details)
for the solution semigroup \eqref{2.semigroup} associated with the Cahn-Hilliard-Oono equation. After that one can also study the upper and lower bounds for its Kolmogorov's $\eb$-entropy, etc. Since all these things are more or less straightforward nowadays (when the key dissipative estimate is obtained, of course, see \cite{Ba,BV1,MZ,Te,Z1,Z2,Z3,Z4} and references therein), we prefer not to give more details here.
\end{remark}

\section{Cahn-Hilliard equation with singular potentials}\label{s5}

In the previous sections, we have considered the case when the nonlinearity is {\it regular} $f\in C^2(\R)$.
In this section, we briefly consider the case of the so-called {\it singular} potentials where the nonlinearity $f$
is defined on the interval $(-1,1)$ only and has singularities at $u=\pm1$, a situation which is currently of great interest, see \cite{CMZ,EKZ,EMZ1}
and references therein. The typical example here is the so-called logarithmic potential
\begin{equation}\label{u.log}
f(u)=\log\frac{1-u}{1+u}-\alpha u
\end{equation}
or the polynomial singularity
\begin{equation}\label{u.pk}
f(u)=\frac u{(1-u^2)^l}-\alpha u,
\end{equation}
where $l>0$.
\par
In this case, it is additionally assumed that the solution $u(t,x)$ is always inbetween minus and plus one:
\begin{equation}\label{using}
-1<u(t,x)<1\ \ \text{for almost all}\ \ (t,x)\in\R_+\times\R^3
\end{equation}
and therefore $f(u(t,x))$ has a sense.
\par
Following \cite{EKZ} (see also \cite{CMZ,D,MZ1} and references therein), we assume that the nonlinearity $f$ satisfies
\begin{equation}\label{fsing}
\begin{cases}
1.\ \ f\in C^2(-1,1),\ \ f(0)=0;\\
2.\ \ \lim_{u\pm\infty}f(u)=\pm\infty;\\
3.\ \ \lim_{u\pm\infty}f'(u)=+\infty\\
\end{cases}
\end{equation}
and, exactly as in the case of regular potentials, we assume that $g\in L^6_b(\R^3)$.
\par
However, in contrast to the case of bounded or cylindrical domains, assumptions \eqref{fsing} look insufficient to derive the key a
priori estimate (at least using the method developed above). Indeed, the third assumption of \eqref{0.1} which connects the growth
rate of $f(u)$ and its antiderivative $F(u)$ has been essential in the derivation of that estimate. But this assumption is clearly
wrong for the case of singular potentials where $f(u)$ is growing {\it faster} than $F(u)$ as $u\to\pm1$. In particular,
for the case of nonlinearity  \eqref{u.log} as well as nonlinearity \eqref{u.pk} with $l<1$, the potential $F(u)$ is
bounded near $u=\pm1$, so $f(u)$ cannot be reasonably estimated through $F(u)$ near the singularities and,
by this reason, we are unable to treat these cases. But if $l>1$, the nonlinearity \eqref{u.pk} obviously satisfies
\begin{equation}\label{f-st-sing}
|f(u)|\le \beta|F(u)|^\kappa+C
\end{equation}
for some positive $\beta$ and $C$ and some $\kappa\in(1,\infty)$ (for \eqref{u.pk}, we have $\kappa=1+\frac1{l-1}$). As shown in the next theorem this assumption is enough in order to obtain the analogues of theorems \ref{Th0.1} and \ref{Th0.1d} for the case of singular potentials.

\begin{theorem}\label{Th.critical} Let the assumptions \eqref{fsing} and \eqref{f-st-sing} hold and $g\in L^6_b(\R^3)$. Then, for every $u_0\in\Phi_b$, the Cahn-Hilliard equation \eqref{CH} possesses at least one global solution $u(t)$, $t\ge0$, (in the sense of \eqref{0.3} plus the extra assumption \eqref{using}) which satisfies the following analogue of \eqref{0.b}:
\begin{multline}\label{0.bsing}
\|u(t)\|_{W^{1,2}_b}^2+\|F(u(t))\|_{L^1_b}+\|\Nx\mu\|_{L^2_b([0,t]\times\R^3)}^2\le\\\le C(1+ t^{3\kappa+1})\(1+\|g\|_{L^6_b}^2+ \|u_0\|_{W^{1,2}_b}^2+\|F(u_0)\|_{L^1_b}\)^{3\kappa-1/2},
\end{multline}
where $\kappa$ is the same as in assumption \eqref{f-st-sing}.
\end{theorem}
\begin{proof} As in the proof of Theorem \ref{Th0.1}, we restrict ourselves to the formal derivation of the key estimate \eqref{0.bsing} and the existence of a solution can be then obtained in a standard way, see \cite{EZ,MZ}. The derivation of this estimate is also similar to what we have done in the proof of Theorem \ref{0.1}, however, the weight function \eqref{0.5} is no longer appropriate and we should use more general weights $\phi_{\eb}(x)$ defined in \eqref{weight} with the parameter
$$
\gamma=3+\frac2{2\kappa-1},
$$
where $\kappa$ is the same as in assumption \eqref{f-st-sing}.
\par
Indeed, multiplying equation \eqref{CH} by $\phi_\eb\mu=\phi_{\eb,x_0}(x)\mu(t)$ where $\phi_{\eb,x_0}(x)=\phi_{\eb}(x-x_0)$ and $\phi_\eb$ is defined by \eqref{weight} (with the parameters $\eb$  being specified below), and arguing exactly as in the proof of Theorem \ref{Th0.1}
(as it is not difficult to see, Lemmas \ref{Lem0.1} and \ref{Lem0.2} remain true for the singular potentials, so no difference so far), we obtain the following analogue of estimate \eqref{0.better}:
\begin{multline}\label{0.better-sing}
\frac{d}{dt} E_{\phi_\eb}(t) + \beta\((|\nabla_x \mu|^2, \phi_\eb) + \| \nabla_x(\phi_\eb^{1/2} \mu) \|^2_{L^2} + (\phi_\eb^3, |f(u)|^6)^{\frac{1}{3}}\)\le\\
\le C\eb^2(\phi_\eb, |\nabla_x u|^2)  + 1)+C\eb^2(\phi_\eb^{1+2/\gamma},|f(u)|^2)+ C\eb^{-3}(\|g\|^6_{L^6_b}+1) ,
\end{multline}
where the weighted energy $E_{\phi_\eb}$ is defined by \eqref{0.wen} and satisfies \eqref{0.weq} (the exponent $7/5$ in the second term of the right-hand side of \eqref{0.better} is now replaced by $1+2/\gamma$ due to the choice of a different weight function, see  \eqref{0.6}).
\par
However, in order to estimate the second term in the right-hand side of \eqref{0.better-sing}, we now need to modify \eqref{0.interp} interpolating between $L^{1/\kappa}_{\phi_\eb}$ and $L^6_{\phi_\eb^3}$ (instead of $L^1_{\phi_\eb}$ and $L^6_{\phi_\eb^3}$). Namely, using the elementary fact that
$$
1+\frac2\gamma=\frac{4\kappa}{6\kappa-1}+3\cdot\frac{2\kappa-1}{6\kappa-1}
$$
together with the H\"older and Young inequalities, we see that
\begin{multline}\label{intnew}
C\eb^2(\phi_\eb^{1+2/\gamma},|f(u)|^2)=C\eb^2([\phi_\eb|f(u)|^{1/\kappa}]^{\frac{4\kappa}{6\kappa-1}},[\phi_\eb^3|f(u)|^6]^{\frac{2\kappa-1}{6\kappa-1}})
\le\\\le C\eb^2(\phi_\eb,|f(u)|^{\frac1\kappa})^{\frac{4\kappa}{6\kappa-1}}(\varphi_\eb^3,|f(u)|^6)^{\frac{2\kappa-1}{6\kappa-1}}=\\=
C\(\eb^{6\kappa-1}(\phi_\eb,|f(u)|^{\frac1\kappa})^{2\kappa}\)^{\frac2{6\kappa-1}}\((\varphi_\eb^3,|f(u)|^6)^{1/3}\)^{\frac{6\kappa-3}{6\kappa-1}}
\le\\\le C\eb^{6\kappa-1}(\phi_\eb,|f(u)|^{1/\kappa})^{2\kappa}+\beta(\varphi_\eb^3,|f(u)|^6)^{1/3}.
\end{multline}
Inserting this estimate into the right-hand side of \eqref{0.better-sing} and using \eqref{f-st-sing} and \eqref{0.weq}, we arrive at the following analogue of inequality \eqref{0.good}
\begin{equation}\label{0.ggood}
\frac d{dt}E_{\phi_\eb}(t)+\beta\|\Nx\mu\|_{L^2_{\phi_\eb}}^2\le C\eb^2 E_{\phi_\eb}(t)+C\eb^{6\kappa-1}[E_{\phi_\eb}(t)]^{2\kappa}+C\eb^{-3}(\|g\|_{L^6_b}^2+1).
\end{equation}
As in the proof of Theorem \ref{Th0.1}, this inequality implies the desired estimate \eqref{0.bsing}. Indeed, introducing $V_\eb(t):=E_{\phi_\eb}(t)$ and eliminating the first term in the right-hand side via the Young inequality, we end up with
$$
\frac d{dt}V_\eb+\beta\eb^{3}\|\Nx \mu\|^2_{L^2_{\phi_\eb}}\le \eb^2V_\eb^{2\kappa}+C(\|g\|_{L^6_b}^2+1),\ V_\eb(0)\le C(1+\|g\|_{L^6_b}^2+\|u_0\|_{\Phi_b}).
$$
As in the proof of Theorem \ref{Th0.1}, we conclude that
\begin{equation}\label{0.esttt}
V_\eb(t)\le 2C(T+1)(1+\|g\|^2_{L^6_b}+\|u_0\|_{\Phi_b}),\ \ t\in[0,T]
\end{equation}
if  $\eb=\eb(T,u_0,g)$ is fixed by
\begin{equation}\label{0.ebbbb}
\eb:=\frac1{[2(T+1)]^\kappa [C(1+\|g\|^2_{L^6_b}+\|u_0\|_{\Phi_b})]^{\kappa-1/2}}.
\end{equation}
Thus,
\begin{equation*}
E_{\phi_{\eb,x_0}}(T)\le \eb^{-3}V_\eb(T)\le C(T+1)^{3\kappa+1}(1+\|g\|^2_{L^6_b}+\|u_0\|_{\Phi_b})^{3\kappa-1/2}
\end{equation*}
and the desired estimate \eqref{0.bsing} follows now by applying the supremum over $x_0\in\R^3$ and using \eqref{2.w-un}. Theorem \ref{Th.critical} is proved.
\end{proof}

The next theorem gives the analogue of Theorem \ref{Th.critical} for the Cahn-Hilliard-Oono equation with singular potentials.

\begin{theorem}\label{Th.chos} Let the assumptions \eqref{fsing} and \eqref{f-st-sing} hold and $g\in L^6_b(\R^3)$. Then, for every $u_0\in\Phi_b$, the Cahn-Hilliard-Oono equation \eqref{CHd} possesses at least one global solution $u(t)$, $t\ge0$, (in the sense of \eqref{0.3} plus the extra assumption \eqref{using}) which satisfies the following analogue of \eqref{0.b}:
\begin{equation}\label{0.bbsing}
\|u(t)\|_{W^{1,2}_b}^2+\|F(u(t))\|_{L^1_b}+\|\Nx\mu\|_{L^2_b([t,t+1]\times\R^3)}^2\le Q(\|u_0\|_{\Phi_b})e^{-\sigma t}+Q(\|g\|_{L^6_b}),
\end{equation}
where the monotone increasing function $Q$ and positive constant $\sigma$ are independent of $u_0$ and $t$.
\end{theorem}
The proof of this theorem repeats almost word by word the proof of Theorem \ref{Th0.1d}. The only difference is that we should use $\gamma=3+\frac2{2\kappa-1}$ instead of $\gamma=5$ in the definition of the weight function $\phi_{\eb(t)}(x)$ and use the refined interpolation inequality \eqref{intnew} instead of \eqref{0.interp}. For this reason, we do not present it here.

\begin{remark}\label{Rem.un-sing} The uniqueness Theorem \ref{Th1.unique}  can be also extended to the singular case. However, this requires to control the derivative $f'(u)$ through $f(u)$ or $F(u)$ and assumptions \eqref{1.psi} are again not compatible with singular potentials and must be modified. For instance, if we assume that
\begin{equation}\label{singgg}
|f'(u)|\le [\Psi(u)]^{\kappa_1},\ \ \Psi(u)\le C_1f(u)+C_2,\ \ \kappa_1<8/5
\end{equation}
for some convex function $\Psi$ and positive $C_1$ and $C_2$, then arguing as in \cite{EKZ} (see Theorem 3.4), we may establish the uniqueness as well as the further regularity of a solution and verify, in particular, that the solution $u$ becomes separated from singularities for positive times ($\|u(t)\|_{L^\infty}\le1-\delta$, for some $\delta>0$). After that the further investigation of the problem can be constructed exactly as for the case of regular potentials.
\par
Note also that condition \eqref{singgg} is {\it stronger} than  \eqref{f-st-sing} which we need for the global existence of a solution. In particular, for the nonlinearities \eqref{u.pk}, we need $k>5/3$ (instead of $k>1$).
\end{remark}

\end{document}